\documentclass{amsart}

\usepackage{amssymb,amsthm,amsmath,amstext,amsxtra}
\usepackage{booktabs}
\usepackage{hyperref}
\hypersetup{colorlinks=true,urlcolor=blue,citecolor=blue,linkcolor=blue}
\usepackage{enumerate}
\usepackage{enumitem}
\usepackage{comment}
\usepackage{colonequals}
\usepackage{float} 
\usepackage{subcaption}

\usepackage[linesnumbered,lined]{algorithm2e}

\usepackage[numbers]{natbib}

\defcitealias{sage}{\scshape SageMath}

\numberwithin{equation}{section}

\theoremstyle{plain}
\newtheorem{theorem}[equation]{Theorem}

\newtheorem{proposition}[equation]{Proposition}
\newtheorem{lemma}[equation]{Lemma}
\newtheorem{corollary}[equation]{Corollary}

\theoremstyle{definition}
\newtheorem{definition}[equation]{Definition}

\theoremstyle{remark}
\newtheorem{remark}[equation]{Remark}

\newcommand{\Z}{\mathbb{Z}}
\newcommand{\Q}{\mathbb{Q}}

\newcommand{\F}{\mathbb{F}}

\newcommand{\PP}{\mathbb{P}}

\newcommand{\calC}{\mathcal{C}}
\newcommand{\tildecalC}{\widetilde{\calC}}

\newcommand{\Otilde}{\widetilde{O}}

\DeclareMathOperator{\Frob}{Frob}

\newcommand{\id}{\mathrm{id}}

\newcommand{\dd}{\mathrm{d}}
\newcommand{\MW}{{\textrm{MW}}}

\newcommand{\psmod}[1]{~(\textup{\text{mod}}~{#1})}

\usepackage{xcolor}
\definecolor{darkred}{HTML}{CC1F1F}
\definecolor{green}{rgb}{.4,.7,.4}
\definecolor{blue}{rgb}{.2,.6,.75}
\definecolor{pastelb}{HTML}{3333FF}
\definecolor{melon}{HTML}{F89EDD}

\definecolor{pastelyellow}{rgb}{0.992157, 0.552941, 0.235294}
\definecolor{pastelorange}{rgb}{0.941176, 0.231373, 0.12549}
\definecolor{pastelred}{rgb}{0.741176, 0., 0.14902}
\definecolor{darkbrown}{rgb}{0.25098, 0., 0.0745098}

\newcommand\blfootnote[1]{%
	\begingroup
	\renewcommand\thefootnote{}\footnote{#1}%
	\addtocounter{footnote}{-1}%
	\endgroup
}

\hypersetup{colorlinks=true,linkcolor=blue,anchorcolor=blue,citecolor=blue}
\hypersetup{pdftitle={Computing Zeta Functions of Cyclic covers in Large Characteristic},
 pdfauthor={Vishal Arul, Edgar Costa, Richard Magner, and Nicholas Triantafillou}}
\author[V. Arul]{Vishal Arul}
\address{Department of Mathematics, Massachusetts Institute of Technology, Cambridge, MA 02139, USA}
\email{varul@mit.edu}
\urladdr{\url{http://math.mit.edu/~varul/}}

\author[A. Best]{Alex J. Best}
\address{Department of Mathematics and Statistics, Boston University, Boston, MA 02215, USA}
\email{alex.j.best@gmail.com}
\urladdr{\url{https://alexjbest.github.io/}}

\author[E. Costa]{Edgar Costa}
\address{Department of Mathematics, Massachusetts Institute of Technology, Cambridge, MA 02139, USA}
\email{edgarc@mit.edu}
\urladdr{\url{https://edgarcosta.org}}

\author[R. Magner]{Richard Magner}
\address{Department of Mathematics and Statistics, Boston University, Boston, MA 02215, USA}
\email{rmagner@bu.edu}
\urladdr{\url{http://math.bu.edu/people/rmagner/}}

\author[N. Triantafillou]{Nicholas Triantafillou}
\address{Department of Mathematics, Massachusetts Institute of Technology, Cambridge, MA 02139, USA}
\email{ngtriant@mit.edu}
\urladdr{\url{https://math.mit.edu/~ngtriant/}}

\thanks{
The authors are grateful to the organizers of Sage Days 87, where this project began.
We would also like to thank the reviewers for their many helpful comments.
The second author was supported by the Simons Collaboration Grant \#550023.
The third author was partially supported by the Simons Collaboration Grant \#550029.
The fifth author was supported by the National Science Foundation Graduate Research Fellowship under Grant \#1122374.
}

\begin{document}

\title[Zeta Functions of Cyclic covers in Large Characteristic]{Computing Zeta Functions of Cyclic covers in Large Characteristic}

\begin{abstract}
  We describe an algorithm to compute the zeta function of a cyclic cover of the projective line over a finite field of characteristic $p$ that runs in time $p^{1/2 + o(1)}$.
  We confirm its practicality and effectiveness by reporting on the performance of our \textsc{SageMath} implementation on a range of examples.
  The algorithm relies on Gon\c{c}alves's generalization of Kedlaya's algorithm for cyclic covers, and Harvey's work on Kedlaya's algorithm for large characteristic.
\end{abstract}

\maketitle

\blfootnote{2010 {\normalfont\itshape
		Mathematics Subject Classification}: 11G20 (primary)
	11Y16, 11M38, 14G10 (secondary)}

\section{Introduction}
\label{sec:introduction.tex}
For $\calC$ an algebraic curve of genus $g$ over a finite field $\F_q$ of characteristic $p$ and cardinality $q = p^n$, the zeta function  of $\calC$ is defined by
\begin{equation*}
  Z(\calC , t) \colonequals \operatorname{exp}\left( \sum_{i = 1} ^\infty \#
	\calC( \F_{q^i} ) \frac{t^i}{i} \right)  = \frac{L( \calC, t)}{(1 -t) (1 - q t)},
\end{equation*}
where $L(\calC, t) \in 1 + t \Z[t]$ is a degree $2g$ polynomial, with reciprocal
roots of complex absolute value $q^{1/2}$, and satisfies the functional equation
$L( \calC, t) = q^g t^{2 g} L( \calC, 1/(t q))$.
In this paper, we address how to effectively compute $Z(\calC, t)$ for a cyclic
cover of $\PP^1$ defined by $y^r = \overline{F}(x)$, where $\overline{F}(x)$ is
squarefree and $p$ is large in comparison to $g$, without any restrictions on
$r$ and $\deg \overline{F}$ sharing a common factor.

For curves of small genus, Schoof's method and its
variants \cite{schoof-85, pila-90,
gaudry-schost-04, gaudry-kohel-smith-11, gaudry-schost-12} can compute $Z(\calC , t)$ in time and
space polynomial in $\log q$ and exponential in the genus.  However, the
practicality of these methods has only been shown for genus at most 2.  These
are known as $\ell$-adic methods, as their efficiency derives from the
realization of the $\ell$-adic cohomology of the variety via torsion points.

Alternatively, Kedlaya \cite{kedlaya-01} showed that $Z (\calC, t)$ can be
determined in quasi-linear time in $p$ for an odd hyperelliptic curve, i.e., $r
= 2$ and $\deg \overline{F} = 2 g + 1$, by computing an approximation of the
Frobenius matrix acting on $p$-adic cohomology (Monsky--Washnitzer cohomology).
Kedlaya's algorithm and its variants are known as $p$-adic methods.  In \cite{harvey-07}, Harvey improved the
time dependence in $p$ to $p^{1/2 + o(1)}$. In \cite{harvey-14}, this
improvement plays a major role in Harvey's algorithm for computing the $p$-local
zeta functions of an odd hyperelliptic curve over $\Z$ for all $p$ up to some
bound.  Kedlaya's original algorithm has been subsequently generalized several
times, for example to superelliptic curves \cite{gaudry-gurel-01},
$C_{a,b}$ curves \cite{denef-vercauteren-06}, even degree hyperelliptic curves
\cite{harrison-12}, and nondegenerate curves \cite{castryck-denef-vercauteren-06}.
More recently,
Gon\c{c}alves \cite{goncalves-15} extended Kedlaya's algorithm to cyclic covers
of $\PP^1$ and Tuitman \cite{tuitman-16, tuitman-17} to general covers. All
these generalizations kept the quasi-linear time dependence in $p$.  Minzlaff \cite{minzlaff-10} improved
Gaudry--G\"{u}rel's algorithm for superelliptic curves by incorporating Harvey's work, giving a $p^{1/2 + o(1)}$ time algorithm.
The algorithms described
above are efficient in practice, and have been integrated into the current
versions of \textsc{Magma} \cite{magma} and \textsc{SageMath} \cite{sage}.

In this paper, we build upon Gon\c{c}alves, Harvey, and Minzlaff's work to
obtain a practical $p^{1/2 + o(1)}$ algorithm for cyclic covers of $\PP^1$.
Theoretically, we already knew of the existence of algorithms with such a time dependence on \(p\) (and their
average polynomial time versions) for arbitrary schemes (see \cite{harvey-15}).
These algorithms for arbitrary schemes have never been implemented, and it is unclear if they can be made to work in practice.
Our algorithm improves the dependence on other parameters over these very general algorithms and provides a step towards a practical average
polynomial time in higher genus, analogous to the progression from $p^{1/2 +
o(1)}$ to average polynomial time for odd hyperelliptic curves by Harvey.

More recently, Tuitman \cite{tuitman-18} combined Harvey's ideas with a deformation approach to give a $p^{1/2 + o(1)}$ algorithm for computing zeta functions of generic projective hypersurfaces of higher dimension.
Tuitman's algorithm has a similar theoretical dependence on the degree of the curve and the degree of the field (over $\mathbb F_{p}$) as our algorithm.

Throughout we will use a bit complexity model for computation and the notation \(\Otilde(x) = \bigcup_k O(x \log^k(x))\).
Our main result is then as follows:
\begin{theorem} \label{thm:mainthm}
  Let $\calC$ be a cyclic cover of $\PP^1$, of genus $g$, defined by
  \begin{equation*}
    \calC: y^r = \overline{F}(x),
  \end{equation*}
  where $\overline{F} \in \F_q[x]$ is a
  squarefree polynomial of degree $d$.
  Let $\tildecalC$ be the curve obtained from $\calC$ by removing the $\delta$
  points at infinity and the $d$ points on the $x$-axis corresponding to the
  zeros of $\overline{F}(x)$.
  Let $M_{\epsilon}$ be the matrix
  of Frobenius acting on $B_{\epsilon}$, where $B_{\epsilon}$ is a basis of the
  Monsky--Washnitzer cohomology of $\widetilde{C}$ defined in
  \eqref{def:Beps}.

  Let $N \geq 1$, and assume
\begin{equation}
\label{eqn:assumption}
p > d (N + \epsilon) r \text{ and } r + d \geq 5.
\end{equation}
Then the entries of $M$ are in $\Z_q$ and we may compute $M$ modulo $p^N$ in time
\begin{equation*}
  \Otilde( p^{1/2} N^{5/2} d^\omega r n +  N^4 r d^4 n \log p + N n^2 \log p)
\end{equation*}
	and space 
\begin{equation*}
    O( (p^{1/2} N^{3/2} + r N^2) d^2 n \log p),
  \end{equation*}
  where $\omega$ is a real number such that the matrix arithmetic operations on
	matrices of size $m \times m$ take $\Otilde(m^{\omega})$ ring operations.
\end{theorem}

With the goal of computing $Z(\calC, t)$ we may apply Theorem \ref{thm:mainthm}
with $N = O(n r d)$, for example as in \eqref{eqn:N}, and this gives the following result:
\begin{theorem} \label{thm:fast}
	In the same setup as Theorem \ref{thm:mainthm},  assume $p > d r (\frac{1}{2}
	g n  + \log_{p}(g) + 2)$.
	We can compute the numerator of the zeta function of $\mathcal{C}$ in time
  \begin{equation*}
  \Otilde( p^{1/2} n^{7/2} r^{7 / 2} d^{5/2 + \omega} +  n^5 r^5 d^8 \log p )
  \end{equation*}
	and space
  $
  \displaystyle
    O((p^{1/2} + n^{1/2} r^{3/2} d^{1/2}) n^{5/2} r^{3/2} d^{7/2} \log p).
    $
\end{theorem}

We also provide the following $O(\log p)$ space alternative to
Theorem~\ref{thm:mainthm}; see Remark~\ref{remark:low_space} for more details.
\begin{theorem} \label{thm:low_space}
In the same setup as Theorem \ref{thm:mainthm},
we may we may compute $M$ modulo $p^N$ in time
$\Otilde(p r d^3 N^3 n + n^2 N \log p)$ time and space $O(r d^2 N n \log p)$.
\end{theorem}

In comparison with Minzlaff's work, in all the theorems above we do not put any restrictions on $r$ and
$\deg(\overline{F})$ sharing a common factor.
Theorem~\ref{thm:low_space} reduces the space complexity of \cite[Proposition 5.1]{goncalves-15} from quasi-linear to logarithmic.
Theorem~\ref{thm:fast} reduces both time and space complexity of \cite[Proposition 5.1]{goncalves-15}
from quasi-linear in $p$ to $p^{1/2 + o(1)}$.
Moreover, we provide a \textsc{SageMath} implementation of our algorithm for computing zeta functions \cite{sagecode}.

As with all adaptations of Kedlaya's algorithm, the heart of our algorithm is a
procedure for computing a $p$-adic approximation to the action of Frobenius on a
well-chosen basis for (a slight modification of) the Monsky--Washnitzer
cohomology of $\calC$. This is described in Lemma~\ref{lemma:sparsefrob}.

The remainder of the paper is organized as follows.
In Section~\ref{sec:setup}, we recall the relevant definitions for Monsky--Washnitzer cohomology.
In Section~\ref{sec:frob_action}, we compute a `sparse' formula for the action of Frobenius on the basis $B_{\epsilon}$.
The formula from Section~\ref{sec:frob_action} includes terms of large positive $x$-degree and large
negative $y$-degree.
Sections~\ref{sec:horizontal} and \ref{sec:vertical} show how to replace terms with cohomologous terms with $x$- and $y$-degree closer to
zero by `horizontal' and `vertical' reductions.
Section~\ref{sec:algorithm} collects the full algorithms, including complexity statements.
We close by demonstrating the practicality of our  implementation in Section~\ref{sec:sample_computations}.

\section{Setup and notation}
\label{sec:setup}
Let $p$ be a prime and let $q = p^n$ for some $n \geq 1$.
Let $\F_q$ and $\F_p$ be the finite fields with $q$ elements and $p$ elements.
We write $\Q_q$ for the unramified extension of degree $n$ of $\Q_p$, and $\Z_q$
for its ring of integers.

We will work under the assumption that \eqref{eqn:assumption} holds.

Let $\overline{F}(x) \in \F_q[x]$ be a polynomial of degree $d$ with no multiple roots.
To $\overline{F}(x)$ we can associate an $r$-cyclic cover of the projective line
$\calC$ defined by
\begin{equation}
\calC\colon y^r = \overline{F}(x).
\end{equation}
Write $\delta \colonequals \gcd(r, d)$.
Then the genus of $\calC$ is
$g = \frac{1}{2} ((d - 1)(r - 1) - (\delta - 1) ) .$
The curve $\calC$ is naturally equipped with an automorphism of order $r$ defined by
\begin{equation}
\rho_r \colon (x, y) \longmapsto (x,  \zeta_r y)
\end{equation}
where  $\zeta_r$ is a primitive $r$-th root of unity in a fixed algebraic closure of $\F_q$.

As in Kedlaya's original algorithm \cite{kedlaya-01}  we pick an arbitrary lift
$F(x) \in \Z_q[x]$ of $\overline{F}(x)$, also 
of degree $d$.
Let $\tildecalC$ be the curve obtained from $\calC$ by removing the $\delta$
points at infinity and the $d$ points on the $x$-axis corresponding to the zeros of $\overline{F}(x)$.
Let
$\overline{A} = \F_q[x, y, y^{-1}]/(y^r - \overline{F}(x))
$
denote the coordinate ring of $\tildecalC$, and write
\begin{equation}
A = \Z_q[x,y,y^{-1}]/(y^r - F(x))
\end{equation}
for the lift of $\overline{A}$ associated to $F(x)$.
Let $A^\dagger$ be the weak completion of $A$, i.e.,
\begin{equation}
A^\dagger = \Z_q ^\dagger  [[x,y,y^{-1}]]/(y^r - F(x)),
\end{equation}
where $ \Z_q ^\dagger  [[x,y,y^{-1}]]$ is the ring of power series whose radius of convergence is greater than one.
We lift the $p$-power Frobenius on $\F_q$ to $A^\dagger$ as follows.
On $\Z_q$, we take the canonical Witt vector Frobenius and set $\sigma(x) \colonequals x^p$.
We then extend $\sigma$ to $A^\dagger$ by the formula
\begin{equation}
\label{eqn:sigmay}
\sigma(y^{-j}) \colonequals y^{-j p} \sum_{k = 0} ^{+\infty} \binom{-j/r}{k}  \bigl(\sigma(F(x)) - F(x)^p \bigr)^k y^{-  k p r} .
\end{equation}
The above series converges (because $p$ divides $\sigma(F(x)) - F(x)^p$) and the
definitions ensure that $\sigma$ is a semilinear  (with respect to the Witt vector Frobenius)
endomorphism of $A^\dagger$.
We extend it to differential forms by $\sigma( f \dd g) \colonequals \sigma(f) \dd(\sigma(g))$.

In the spirit of Kedlaya's algorithm, we determine the zeta function of $\calC$ by
computing the Frobenius action on subspace
of $H^1 _\MW( \tildecalC)$ spanned by the set
\begin{align}
\label{def:Beps}
B_{\epsilon} = \left\{ x^{i} \frac{\dd x}{y^{j + \epsilon r}} :\ i \in \{0,\ldots, d -2\}, j \in \{1,\ldots, r - 1\} \right\}, && \text{where } 
\epsilon = \begin{cases}
0 & \text{if }\delta = 1 \\
1 & \text{if }\delta > 1.
\end{cases}
\end{align}
This subspace is Frobenius stable and
$0$ is the only element fixed by the induced automorphism $\rho_r$.
When $\delta > 1$, using the basis $B_{1}$ allows us to avoid divisions by zero
while reducing differentials (cf. Lemma~\ref{lemma:0denominator}). This
is critical for generalizing Harvey's work to this setting.

If $\eta\colon \langle B_{\epsilon} \rangle \rightarrow H^1 _\MW( \calC)$ is the projection map, then we have
\begin{equation}
\langle B_{\epsilon} \rangle = H_{\MW}^{1}(\calC)
\oplus \ker(\eta).
\end{equation}
where $ \ker(\eta)$ is a $\delta - 1$ dimensional vector space stable under Frobenius.
Thanks to Gon\c{c}alves's work \cite[Proof of Theorem 7.5]{goncalves-15}, we have an explicit description for
the characteristic polynomial $U(t) \colonequals \det(t \cdot \id - \Frob_q | \ker(\eta))$ of Frobenius acting on $\ker(\eta)$:
\begin{equation}
  U(t) \colonequals \det(t \cdot \id - \Frob_q | \ker(\eta)) = \det(t \cdot \id - P ) \cdot (t - 1)^{-1}
\end{equation}
where the matrix $P$ represents the permutation induced by $q$-th power Frobenius action on the roots of $T^\delta - f_d$, where $f_d$ is the leading term of $\overline{F}(x)$.
In the case that $\overline{F}(x)$ is monic the expression above simplies to $U(t) = \prod_{i \mid \delta, i > 1} \bigl( t^{k_i} - 1 \bigr)^{\frac{\varphi(i)}{k_i}},$
where $k_i$ is the order of $q$ in $\bigl( \Z / i \Z \bigr) ^{\times}$.
Thus our goal is to compute a $p$-adic approximation of the matrix $M_\epsilon$ representing $\sigma$ with respect to $B_\epsilon$.

\section{The Frobenius action on differentials}
\label{sec:frob_action}
We now rewrite the Frobenius expansion of a basis element in a sparse way where the number of terms does not depend on $p$.
This is a generalization of \cite[Proposition 4.1]{harvey-07} and \cite[Proposition 4.1]{minzlaff-10}, which is made possible due to the analysis performed by Gon\c{c}alves in \cite[\S 6]{goncalves-15}.

\begin{lemma}
  \label{lemma:sparsefrob}
  Let $N > 0$ be a positive integer, $0 \leq i \leq d - 2$ and $\epsilon r + 1 \leq j \leq (1 + \epsilon)r - 1$. Suppose $p > d(N+\epsilon)r$ and $x^i y^{-j} \dd x \in B_{\epsilon}$. 
  For $0 \leq \ell < N$, write
  \begin{equation}
    D_{j,\ell} \colonequals  \sum_{k= \ell}^{N-1}   (-1)^{k - \ell}  \binom{-j/r}{k} \binom{k}{\ell}
    \quad \text{and}
    \quad
    \mu_{j,\ell,b} \colonequals p D_{j,\ell} \sigma(F)^\ell _b,
  \end{equation}
  where $\sigma(F)^{\ell} _b$ is the coefficient of $x^{pb}$ in $\sigma(F(x))^\ell$.
  The differentials
  $\sigma( x^i y^{-j} \dd x )$
  and
  \begin{equation}
    \label{eqn:frobact}
    T_{(i,j)} \colonequals
    x^{p(i + 1) -1}  y^{-j p}
    \sum_{\ell = 0} ^{N-1}
    \sum_{b = 0}^{d \ell}
    \mu_{j, \ell, b}
    x^{p b}
    y^{- \ell  p r } \dd x
  \end{equation}
  {differ in cohomology by an element of $p^N \operatorname{span}_{\Z_q}(B_{\epsilon})$.}
\end{lemma}
\begin{proof}
  From \eqref{eqn:sigmay} we obtain
  \begin{equation}
    \sigma( x^i y^{-j} \dd x )  =   \sum_{k = 0} ^{+\infty} \underbrace{ p x ^{p  (i+1) - 1}    \binom{-j/r}{k}  \bigl(\sigma(F(x)) - F(x)^p \bigr)^k y^{- p (j  +  k r)} \dd x}_{\equalscolon U_k} .
  \end{equation}

  We claim that for $k \geq N$ the reductions of $U_k$ lie in $p^N \operatorname{span}_{\Z_q}(B_{\epsilon})$.

  To show this we start by rewriting $U_k$.
  Since $p$ divides $\sigma(F(x)) - F(x)^p$, we have
  \begin{equation}
    U_k = p^{k +1} H(x) y^{- p (j  +  k r)} \dd x
  \end{equation}
  where $H(x) \in \Z_q[x]$ of degree at most
  $pi + p - 1 +  d k p <  p d (k + 1)$.
  Define
  \begin{equation}
    L = \begin{cases}
      p(k + 1) - 1 &\text{if } \epsilon = 0 \\
      \left\lfloor{\frac{p(j + kr)}{r}} \right\rfloor - \epsilon &\text{if }
      \epsilon > 0.
    \end{cases}
  \end{equation}
  Now we will expand $H(x)$ $F$-adically to $L$ terms. Taking $j' \in [1,r]$
  congruent to $pj \pmod{r}$, and applying the relation $F(x) = y^r$, we have
  \begin{equation}
    {
      U_{k} = p^{k+1} \left(G(x) y^{-\epsilon r - j'}  + \sum_{\ell = 0}^{L}  G_\ell(x) y^{r\ell - p(j +  k   r)}  \right) \dd x,
    }
  \end{equation}
  where each $G_\ell (x) \in \Z_q[x]$ has degree at most $d-1$ and $G(x)$ has degree 
  at most
  \begin{equation}
    pd(k+1) - 1 - d L
    \le \begin{cases}
      d - 1 & \text{if } \epsilon = 0 \\
      0 & \text{if } \epsilon > 0.
    \end{cases}
  \end{equation}

  Taking
  $
  \nu = \lfloor\log_p{p(j+kr) - r\ell}\rfloor \leq 1 + \lfloor \log_{p} (k+1 + \epsilon)r \rfloor
  $
  , Gon\c{c}alves \cite[Proposition 6.1]{goncalves-15}  shows that the reduction of
  $
  p^{\nu} G_{\ell}(x) y^{r\ell - p(j +  k   r)}\dd x
  $
  lies in $\operatorname{span}_{\Z_{q}}(B_{\epsilon})$. 

  Similarly, \cite[Proposition 6.2]{goncalves-15} says that taking
  \begin{equation}
    \mu = \left \lfloor \log_p((r(\deg(G) + 1) - (\epsilon r + j') d)/\delta) \right\rfloor \leq 1 + \lfloor\log_p(rd)\rfloor,
  \end{equation}
  the reduction of
  $
  p^{\mu} G(x) y^{-\epsilon r - j'}\dd x
  $
  lies in $\operatorname{span}_{\Z_{q}}(B_{\epsilon})$.

  Since $p > d(N+ \epsilon)r$, both $\mu = 1$ and $\nu \leq 1 + k - N$, so the
  reductions of $U_k$ for $k \geq N$ lie in $p^N
  \operatorname{span}_{\Z_q}(B_{\epsilon})$.

  The lemma follows by the rearranging the truncated series as follows:
  \begin{equation*}
    \begin{aligned}
      \sum_{k = 0} ^{N-1} \binom{-j/r}{k}  \bigl(\sigma(F(x)) - y^{pr} \bigr)^k y^{ -  k p r}
      &=      \sum_{k = 0} ^{N-1}   \sum_{\ell= 0}
      ^{k} (-1)^{k - \ell} \binom{-j/r}{k} \binom{k}{\ell}
      \sigma\bigl(F(x)\bigr)^\ell
      y^{ p r (k - \ell)} y^{- p r k   } \\
      &=  \sum_{\ell = 0} ^{N-1}    \sum_{b=0}^{d\ell} D_{j, \ell}
      \sigma(F)^\ell _b x^{pb} y^{ - \ell p r}.
    \end{aligned}
  \end{equation*}
  \qedhere
\end{proof}

\section{Reducing differentials}
\label{sec:reduction}
{ The powers of $x$ and $y$ appearing in $T_{(i,j)}$ (as in Lemma \ref{lemma:sparsefrob}) are much larger than those appearing in our choice of representatives for the basis $B_{\epsilon}$. We use relations (co-boundaries) coming from the differentials of functions on our curve to `reduce' the terms from $T_{(i,j)}$ to linear combinations of elements of $B_{\epsilon}$. We proceed in two-stages. Horizontal reduction reduces the $x$-degree while leaving the $y$-pole order constant. Vertical reduction decreases the $y$-pole order without increasing the $x$-degree. Given a differential $\omega$, we call the unique cohomologous differential $\omega' \in \text{span}(B_{\epsilon})$ the \emph{reduction of }$\omega$. We may also abuse notation and call intermediate products of the vertical/horizonal reduction process \emph{reductions} of $\omega$.

	Organizing our work carefully, we can compute the reduction of $\omega$ modulo $p^{N}$ by performing intermediate steps modulo $p^{N+1}$.}

\subsection{Horizontal reductions} \label{sec:horizontal}

We follow the steps of Harvey and Minzlaff.
Decompose $F(x)$ as
$ F(x) = f_d x^{d} + P(x)$, 
where $P(x)$ has degree at most $d - 1$.
\begin{definition}
  For $s \in \Z_{\geq -1}$ and $t\in \Z_{\geq 0} $ define the vector space
  \begin{equation}
    W_{s, t} = \{ G(x) x^{s} y^{-t} \dd x : \deg G \le d - 1 \}
  \end{equation}
  equipped with the standard monomial basis.

  Let
  $M_{H}^{t}(s)  \colon W_{s, t} \to W_{s - 1, t}$
  be the linear map given by the matrix
  \begin{equation}
    M_{H}^{t}(s) =
    \begin{pmatrix}
      0 & 0 & \cdots & 0 & C^t_{0}(s) \\
      D_{H}^{t}(s) & 0 & \cdots & 0 & C^t_{1}(s) \\
      0 & D_{H}^{t}(s) & \cdots & 0 & C^t_{2}(s) \\
      \vdots &  & \ddots &  & \vdots \\
      0 & 0 & \cdots & D_{H}^{t}(s) & C^t_{d - 1}(s)
    \end{pmatrix}
  \end{equation}
  where
  $D_{H}^{t}(s) = (d(t - r) - r s) f_{d}$
  and where $C^t_{h}(s)$ is the coefficient of $x^{h}$ in the polynomial $ C^t(x, s) = r s P(x) - (t - r) x P'(x)$.
  Moreover, for $s_{0} < s_{1}$ we write
  \begin{equation}
    \begin{aligned}
      D^{t}_{H}(s_{0}, s_{1}) &\colonequals D_{H}^{t}(s_{0} + 1) D_{H}^{t} (s_{0} + 2)
      \cdots D_{H}^{t}(s_{1}); \\
      M^{t}_{H}(s_{0}, s_{1}) &\colonequals M_{H}^{t}(s_{0} + 1) M_{H}^{t} (s_{0} + 2)
      \cdots M_{H}^{t}(s_{1}).
    \end{aligned}
  \end{equation}
\end{definition}

\begin{lemma}
  \label{lemma:horizontal}
  For $s \in\Z_{\geq 0}$, $t \in \Z_{\geq 0}$, and $\omega \in W_{s, t}$, we have $D_{H}^{t}(s)\omega \sim M_{H}^{t}(s) \omega$ in cohomology.
\end{lemma}
\begin{proof}
  See \cite[Proposition~5.4]{harvey-07} or \cite[Proposition~5.1]{minzlaff-10}.
  The same algebraic manipulations hold in the cyclic cover setting, as long we
  do not divide by $D_{H}^{t}(s)$, as this might be zero.
\end{proof}

In the case that $d$ and $r$ share a common factor, i.e. $\delta > 1$ and $\epsilon = 1$, then $D_{H}^{t}(s)$ might be identically zero.
The next lemma ensures this cannot happen due to our choice of basis
$B_{\epsilon}$.

\begin{lemma}
  \label{lemma:0denominator}
  We have $D_H^t(s) \neq 0$, while applying horizontal reductions to $T_{(i, j)}$, for $0 \leq i \leq d - 2$ and $1 + \epsilon r \leq j \leq (1 + \epsilon)r - 1$.
\end{lemma}
\begin{proof}
  By inspecting the Frobenius formula \eqref{eqn:frobact} for a fixed value of
  $\ell$, (1) the pole order of $y$ is $t = p (j + r \ell)$, where $1 + \epsilon r \leq j \leq (1 + \epsilon)r - 1$ and (2) the largest power of $x$ is at most $p(d \ell + i + 1) - 1 \le pd(\ell + 1) - 1$. Since the largest power of $x$ in $W_{s,t}$ is $s + d - 1$, we need only consider the case $s + d -1 \le pd(\ell + 1) - 1$.

  If $\delta = 1$, then $\epsilon = 0$ and $d(t - r) - r s
  \equiv d j p \not\equiv 0 \bmod{r}$.

  If $\delta > 1$, then $\epsilon = 1$, so $j \ge 1 + r$ and $t \ge p (1 + r
  (\ell + 1))$. Using $s + d < p d (\ell + 1)$,
  \begin{equation}
    d (t - r) - r s  = dt - r(s+d) \ge dp(1 + r(\ell + 1)) - r(pd(\ell + 1)) = dp > 0. \qedhere
  \end{equation}

  %
  %
\end{proof}
\begin{corollary} \label{cor:0pdenominator}
  In the same setting as Lemma~\ref{lemma:0denominator}, $D_H ^t (s) \equiv 0 \psmod{p}$ if and only if $s \equiv - d \mod p$.
\end{corollary}
\begin{proof}
  As in Lemma~\ref{lemma:0denominator}, the pole order of $y$ is $t = p (j + r \ell)$, thus
  \begin{equation}
    D_{H}^{t}(s) \colonequals	(d (t - r) - r s) f_{d} \equiv -r (d + s) f_{d} \psmod{p}.
  \end{equation}
  By assumption, neither $r$ nor $f_{d}$ are divisible by $p$, so we only divide by $p$ exactly when $s \equiv -d \mod p$.
\end{proof}

\begin{lemma} \label{lemma:nop2denominator}
  Suppose $p > d (N + \epsilon) r$ and $s \equiv -1 \psmod{p}$.
  Then $D_{H}^{t}(s - (d - 1))$ is divisible by $p$, but it is not divisible by $p^{2}$.
\end{lemma}
\begin{proof}
  As $s - (d - 1) \equiv -d \mod{p}$, we know this denominator is divisible by
  $p$. It equals
  $f_{d}(d(t - r) - r(s - (d - 1))) = f_{d}(d t - r s - r)$.
  Since $f_{d}$ is coprime to $p$, we  analyze the piece
  $d t - r (s +	1)$.
  Inspecting the Frobenius formula \eqref{eqn:frobact} and considering that horizontal reduction decreases the exponent of $x$, we see
  \begin{equation}
    \begin{aligned}
      p - 1 &\le s \le p(i + 1) - 1 + p d (N - 1)
            &
      0 &\le i \le d - 2
      \\
      0 &\le t \le j p + (N - 1) p r
        &
      \epsilon r + 1 &\le j \le (1 + \epsilon) r - 1
    \end{aligned}
  \end{equation}
  where $\epsilon \in \{0, 1\}$.
  From these inequalities we obtain
  \begin{equation}
    |d t - r (s + 1)| \le \max
    \{ d t, r (s + 1) \} < d p (N + \epsilon) r < p^{2},
  \end{equation}
  thus the denominator has
  $p$-valuation exactly 1.
\end{proof}

Now we describe the horizontal reduction procedure in a fashion similar to that
in Harvey's \cite[\S 7.2]{harvey-07}.
Following the notation of \eqref{eqn:frobact}, let $v_\ell$ be a vector representing a differential form in $W_{p \ell - 1, t}$ that is cohomologous to
\begin{equation}
  \sum_{b \geq \ell }^{d k}
  \mu_{j, k, b - i - 1}
  x^{p b - 1}
  y^{- t} \dd x,
  \quad \text{where }t = p (k r + j).
\end{equation}

As in Harvey \cite[\S 7.2]{harvey-07}, we say a vector is \emph{1-correct} if the first coordinate (corresponding to the highest power of $x$) is both $0$ modulo $p$ and correct modulo $p^{N+1}$, and the other coordinates are correct modulo $p^N$.

Given $v_\ell$ which is $1$-correct, we show how to compute $v_{\ell-1}$ which is also $1$-correct.
First we get down to $W_{\ell p - d - 1, t}$, by doing the first $d$ reductions modulo $p^{N+1}$, as follows:
\begin{equation}
  \begin{aligned}
    v_{\ell}^{(1)} &= v_{\ell}
                   &
    \in& W_{\ell p - 1, t}
    \\
    v_{\ell}^{(2)} &= D_{H}^{t}(\ell p - 1)^{-1} M_{H}^{t}(\ell p - 1) v_{\ell}^{(1)}
                   &
    \in& W_{\ell p - 2, t}
    \\
    \vdots
    &&
    \vdots
    \\
    v_{\ell}^{(d + 1)} &= D_{H}^{t}(\ell p - d)^{-1} M_{H}^{t}(\ell p - d) v_{\ell}^{(d)}
                       &
    \in& W_{\ell p - d - 1, t}.
  \end{aligned}
\end{equation}
Then we get down to $W_{(\ell - 1) p, t}$ via
\begin{equation}
  v_{\ell}' = D_{H}^{t}( (\ell - 1) p, \ell p - d - 1)^{-1} M_{H}^{t}( (\ell - 1) p, \ell p - d -
  1) v_{\ell}^{(d + 1)},
\end{equation}
and then finally
\begin{equation}
  v_{\ell - 1} = \mu_{j, \ell, (\ell -1) - i - 1} x^{p (\ell - 1) - 1}
  y^{- t} \dd x
  + D_{H}^{t}((\ell - 1) p)^{-1} M_{H}^{t}( (\ell - 1) p)
  v'_{\ell}\text{.}
\end{equation}

An analysis similar to Harvey's \cite[\S 7.2.2]{harvey-07} shows that all coefficients of $M_{H}^{t}(\ell p - d) v_{\ell}^{(d)}$ are divisible by $p$ and correct modulo $p^{N+1}$. Then, Lemma \ref{lemma:nop2denominator} implies that $v_{\ell}^{(d+1)}$ is correct modulo $p^{N}$. By Corollary \ref{cor:0pdenominator}, $v_{\ell}'$ is correct modulo $p^{N}$. Since the first row of $M_{H}^{t}((\ell-1)p)$ is zero modulo $p$, the vector $v_{\ell-1}$ is $1$-correct.

Furthermore, we may also speed up the evaluation of $M_{H}^{t}( (\ell-1) p, \ell p
- d - 1)$ and $D_{H}^{t}( (\ell - 1) p, \ell p  - d - 1)$ by $p$-adically
interpolating the remaining values from the first $N$ values. See  \cite[\S
7.2.1]{harvey-07} and Section~\ref{sec:algorithm} for more details.

\subsection{Vertical reductions} \label{sec:vertical}
Vertical reduction replaces differentials with cohomologous differentials with smaller pole order in $y$.
While we performed horizontal reductions by working with $d$-dimensional vector spaces of differential forms,
vertical reductions arise most naturally on $(d-1)$-dimensional vector spaces.

\begin{definition}
  For $t \in \Z_{\geq 0}$ and $j \in \{1,\dots,r-1\}$, define the vector space
 \begin{equation}
 V_{t}^{j} \colonequals W_{-1, rt + j} \cap W_{0, rt + j},
 \end{equation}
 equipped with the standard monomial basis.
\end{definition}

Vertical reduction operates via a series of maps $V^j_{t} \to V^{j}_{t-1}$ which are identity maps in cohomology. To define the maps, we need a lemma.

\begin{lemma} \label{vertredprereq}
	Let $A \in \Z_{q}[x]$ be a polynomial with $\deg(A) < 2d - 1$. Then, there exist unique polynomials $R, S \in \Z_{q}[x]$ such that $\deg(R) < d - 1, \deg(S) < d$,
	and
	$A(x) = R(x)F(x) + S(x)F'(X).$
\end{lemma}

\begin{proof}
Since $F$ is separable and $\overline F$ is squarefree, we can find $R_0$ and $S_0$ such that $1 = R_0F + S_0F'$ by the Euclidean algorithm. Then $A = (AR_0)F + (AS_0)F'$. There is a unique $S$ and $T$ satisfying $AS_0 = TF + S$ and $\deg(S) < d$. Set $R = AR_0 - TF'$. Since $\deg(A) < 2d - 1$ and $\deg(SF') < 2d-1$, it follows that $\deg(RF) < 2d-1$, so $\deg(R) < d-1$.

Uniqueness follows immediately, since the vector spaces of polynomials of degree less than $2d-1$ and of pairs of polynomials of degrees less than $d-1$ and less than $d$ both have dimension $2d - 1$.
\end{proof}
We may now define the vertical reduction maps.
\begin{definition} \label{definevertred}
  For each $i \in \{0, \ldots, d-2 \}$, let $R_i, S_{i} \in \Z_{q}[x]$ be the unique polynomials of $\deg(R_i) < d-1, \deg(S_i) < d$ such that
    \begin{equation}\label{eqn:xiRiFSiFprime}
    x^i = R_i(x) F(x) + S_i(x) F'(x).
  \end{equation}
  Write $(rt - r + j) R_i(x) + r S_i'(x) = \gamma_{i,0} + \gamma_{i,1}x + \cdots + \gamma_{i,d-2}x^{d-2}$.
  Define $M_{V}^{j}(t)$ and $D_{V}^{j}(t)$ by
  \begin{equation}
    \begin{aligned}
      M_{V}^{j}(t) & \colonequals
      \begin{pmatrix}
        \gamma_{0,0} & \gamma_{1,0} & \cdots & \gamma_{d-2,0} \\
        \gamma_{0,1} & \gamma_{1,1} & \cdots & \gamma_{d-2,1} \\
        \vdots &  & \ddots & \vdots \\
        \gamma_{0,d-2} & \gamma_{1,d-2} & \cdots  & \gamma_{d-2,d-2}
      \end{pmatrix}, \\
      D_{V}^{j}(t) & \colonequals rt - r + j.
    \end{aligned}
  \end{equation}
  Further define
  \begin{equation}
  \begin{aligned}
    M_{V}^{j}(t_1,t_2) & \colonequals M_{V}^{j}(t_1 + 1)\cdot M_{V}^{j}(t_1 + 2) \cdots M_{V}^{j}(t_2),\\
    D_{V}^{j}(t_1,t_2) & \colonequals D_{V}^{j}(t_1 + 1)\cdot D_{V}^{j}(t_1 + 2) \cdots D_{V}^{j}(t_2).
  \end{aligned}
  \end{equation}
\end{definition}

\begin{lemma}
	Consider $M_{V}^{j}(t)$ as a linear map from $V^{j}_{t}$ to $V^{j}_{t-1}$ with respect to their standard bases. Then, for any $\omega \in V^{j}_{t}$,
	\begin{equation} \label{lem:vertreduction}
	D_{V}^{j}(t)	\omega \sim M_{V}^{j}(t)\omega
	\end{equation}
	in cohomology. More generally, considering $M_{V}^{j}(t_1,t_2)$ as a linear map from $V^{j}_{t_2}$ to $V^{j}_{t_1}$ with respect to their standard bases, for any $\omega \in V^{j}_{t_2}$,
	\begin{equation} \label{lem:multiplevertreduction}
	D_{V}^{j}(t_1,t_2) \omega \sim M_{V}^{j}(t_1,t_2)\omega.
	\end{equation}
\end{lemma}

\begin{proof}
  For any $S(x) \in \Q_{q}[x]$,
  \begin{equation}
  \begin{aligned}
    0 & \sim  \dd\left( \frac{-r}{rt - r + j} S(x) y^{-(rt - r + j)} \right) \\
    & =   S(x)F'(x)y^{-(rt+j)} \dd x + \frac{-r}{rt-r+j}S'(x)y^{-(rt-r+j)} \dd x.
  \end{aligned}
\end{equation}
    So, writing $x^i = R_i(x)F(x) + S_i(x)F'(x)$ as in \eqref{eqn:xiRiFSiFprime}, we have
  \begin{equation*}
    \label{eqn:verticalrelation}
    \begin{aligned}
      x^i y^{-(rt + j)} \dd x   & =     R_i(x)F(x)y^{-(rt + j)} \dd x + S_i(x)F'(x)y^{-(rt + j)} \dd x \\
                                &  \sim    R_i(x)y^{-(rt -r + j)} \dd x + \frac{r}{rt-r+j}S_i'(x)y^{-(rt-r+j)}\dd x \\
                                &  =    \frac{(r(t-1) + j)R_i(x) + r S_i'(x)}{r(t-1) + j} y^{-(r(t-1) + j)} \dd x \\
                                & =    (D_{V}^{j}(t_1,t_2))^{-1} \left(\gamma_{i,0} + \gamma_{i,1}x + \cdots + \gamma_{i,d-2}x^{d-2}\right) y^{-(r(t-1) + j)} \dd x.
    \end{aligned}
  \end{equation*}
  From this, \eqref{lem:vertreduction} follows by linearity. Then \eqref{lem:multiplevertreduction} is immediate from \eqref{lem:vertreduction}.
\end{proof}

\begin{remark}
	  If we could work at infinite (or even very large) precision without it costing us
		computation time, this would be sufficient. However, in practice (and in
		theory), working with fewer extra bits results in significant time savings.
		Fortunately, we will see that when $p$ is sufficiently large, the valuations
		of the coefficients of $D_{V}^{j}(t_1,t_2)^{-1} M_{V}^{j}(t_1,t_2)$ are never less
		than $-1$. As a result, given any element of $V^j_{t}$, we will be able
		to compute a cohomologous element of $V^{j}_{0}$ while only losing a single digit of $p$-adic absolute precision.
\end{remark}

		Now, we follow Harvey's lead and study the coefficients of the matrices $M_{V}^{j}(t_1,t_2)$ and scalars $D_{V}^{j}(t_1,t_2)$.
		Lemma \ref{lem:almostintegralityvertical} will be our main technical tool.
		\begin{lemma} \label{lem:almostintegralityvertical}
		Suppose $A \in \Z_{q}[x]$ and  $B, G_{-t_2 + 1}, \dots, G_{-t_1} \in \Q_{q}[x]$ satisfy
		\begin{align} \label{eqn:almostintegralityvertical}
		A(x)y^{-r t_2 - j}\dd x = B(x) y^{-r t_1 - j}\dd x + \dd \left(\sum_{t = -t_2 + 1}^{-t_1} G_t(x) y^{r t - j} \right).
		\end{align}
		Fix $C \in \Z_{q}$. If $\displaystyle \frac{C}{rt_1 + j}, \frac{C}{r (t_1 + 1) + j}, \ldots, \frac{C}{r(t_2 - 1) + j} \in \Z_{q}$ then $C\cdot B(x) \in \Z_{q}[x]$.
		\end{lemma}

		\begin{remark}
		In our setting, $rt_1 + j \leq rt_2 + j < p^2$, so we may take $C = p$. Applying Lemma \ref{lem:almostintegralityvertical} with $A(x) = 1, x, \dots, x^{d-1}$, the coefficients of $p D_{V}^{j}(t_1,t_2)^{-1} M_{V}^{j}(t_1, t_2)$ all belong to $\Z_q$.
		\end{remark}

		We defer the proof of Lemma \ref{lem:almostintegralityvertical} to the end of the section, and collect the consequences needed for our main algorithm.

\begin{lemma} \label{lem:invertiblematrix}
If $r(t-1) \equiv -j\pmod{p}$, then $M_{V}^{j}(t)^{-1}$ is integral.
\end{lemma}
		The proof is identical to the proof of Harvey's \cite[Lemma 7.7]{harvey-07} after replacing each occurrence of $2g$ with $d-1$.
		Indeed, the matrices are the same, up to multiplication by a unit.

\begin{lemma} \label{lem:verticalmatrixzero}
If $rt_1 \equiv -j \mod p$, then $M_{V}^{j}(t_1, t_1+p)$ is zero modulo $p$.
\end{lemma}

\begin{proof}
  Here, the proof generalizes \cite[Lemma 7.9]{harvey-07}.
  By Lemma \ref{lem:almostintegralityvertical},
  \begin{equation}
  X \colonequals p D_{V}^{j}(t_1, t_1 + p + 1)^{-1} M_{V}^{j}(t_1, t_1+p +1) 
  \end{equation}
   has integral coefficients. By a computation similar to Lemma \ref{lemma:nop2denominator}, $D_{V}^{j}(t_1, t_1 + p + 1) = p^2\cdot u$ for some unit $u \in \Z_q^\times$, since the first and last terms contribute exactly one power of $p$ and no other terms contribute. Then,
   \begin{equation*}
   M_{V}^{j}(t_1, t_1+p) = p^{-1} D_{V}^{j}(t_1, t_1 + p + 1) X M_{V}^{j}(t_1 + p + 1)^{-1} = p uX M_{V}^{j}(t_1 + p + 1)^{-1}.
   \end{equation*}
   Lemma~\ref{lem:invertiblematrix} implies $M_{V}^{j}(t_1 + p + 1)^{-1}$ is integral, so $M_{V}^{j}(t_1, t_1+p) \equiv 0 \mod{p}$.
\end{proof}

Lemma \ref{lem:verticalmatrixzero} implies that the matrix $Y := D_V^j(t_1,t_1 + p)^{-1}M_V^j(t_1, t_1 + p)$ is integral when $rt_1 \equiv -j \mod p$. Hence the denominators of ``vertically reductions'' of differentials do not grow, at least if we reduce in appropriate batches of $p$ steps.

Unfortunately, we may not start with $t_1$ satisfying $rt_1 \equiv -j \mod p$. Reducing to this case involves dividing by $p$ at most once. To compensate, we must compute $Y$ to one extra digit of $p$-adic precision.

		Having collected our results, we now prove Lemma \ref{lem:almostintegralityvertical}. Much like Kedlaya's proof of \cite[Lemma 2]{kedlaya-01}, we compare power series expansions of differentials in the uniformizer $y$ near $(\theta_i,0)$ for all roots $\theta_i$ of $F$. We give a full proof for clarity. The argument relies heavily on the following lemma:

		\begin{lemma} \label{lem:integralitytransfers}
		Let $G \in \Q_{q}[x]$ be a polynomial with $\deg(G) < d$. View $G$ as an
		element of $\Q_{q}[x,y]/(y^r - F(x))$. Let $\theta_1, \dots, \theta_d$ be
		the roots of $F$. Let $K_i \cong \Q_{q}{(\!(y)\!)}$ be the fraction field of the
		completion of the local ring at $(\theta_i,0)$. The following are
		equivalent:
		\begin{enumerate}[label=(\roman*)]
		\item $G$ has integral coefficients as a polynomial.
		\item $G$ has integral coefficients as a power series in $K_i$ for all $i$.
		\item The coefficient of $y^0$ of $G$ as a power series in $K_i$ is integral for all $i$.
		\end{enumerate}
		\end{lemma}

		\begin{proof}
	  It is trivial that (ii) implies (iii). \\
	  ``(iii) implies (i)'' follows immediately from the observation that the coefficient of $y^0$ of $G$ as a power series in $K_i$ is equal to $G(\theta_i)$. Since $\deg(G) < d$ and the roots of $F$ are distinct mod $p$, the Lagrange interpolation formula shows that $G \in \Z_{q}[x]$.\\
	  ``(i) implies (ii)'' follows immediately from the fact that $F$ has distinct roots mod $p$, so expanding $x$ as a power series in $y$ in $K_i$ never requires division by a non-unit.
		\end{proof}

		With Lemma \ref{lem:integralitytransfers}, the proof of Lemma
		\ref{lem:almostintegralityvertical} follows from the observation that the
		map $\dd$ commutes with passage to the local ring.

    \begin{proof}[Proof of Lemma \ref{lem:almostintegralityvertical}.]
      Note that for all roots $\theta_i$ of $F$, $F'(\theta_i) \in
      \Z_{q}^{\times}$, since $\overline{F}$ is separable. Then, as power series
      in $y$ (near $(\theta_i,0)$),
      \begin{equation*}
        \begin{aligned}
          A(x)y^{r (-t_2) - j}\dd x  = r A(x) y^{r (-t_2 + 1) - j - 1} F'(x)^{-1} \dd y 
          & = \sum_{t = -t_2 + 1}^{\infty} a_{i,t} y^{r t - j - 1} \dd y, \\
          B(x)y^{r (-t_1) - j}\dd x & = \sum_{t = -t_1 + 1}^{\infty} b_{i,t} y^{r t - j - 1} \dd y,
        \end{aligned}
      \end{equation*}
      where the $a_{i,t}$ are integral by Lemma \ref{lem:integralitytransfers}, but we have no bounds (yet) on the $b_{i,t}$. Then,
      \begin{equation*}
        \dd \left(\sum_{t = -t_2 + 1}^{-t_1} G_t(x) y^{r t - j} \right) = \sum_{t = -t_2 + 1}^{-t_1 } a_{i,t} y^{r t - j - 1} \dd y + \sum_{t = -t_1 + 1}^{\infty} (a_{i,t} - b_{i,t}) y^{r t - j - 1} \dd y.
      \end{equation*}
      Integrating term by term,
      \begin{equation}
          \sum_{t = -t_2 + 1}^{-t_1} G_t(x) y^{r t - j}  = \sum_{t = -t_2 +1}^{-t_1} \frac{a_{i,t}}{r t - j} y^{r t - j} + \sum_{t = -t_1+1}^{\infty} \frac{a_{i,t} - b_{i,t}}{r t - j} y^{r t - j},
      \end{equation}
      In particular, if $C$ satisfies
      $\frac{C}{r\cdot t + r - j} \in \Z_{q}$,  for all $t \in \{-t_2, \dots, -t_1 - 1\},$
      then the coefficients of
      $y^{r (-t_2 + 1) - j}, y^{r (-t_2 + 2) - j}, \ldots, y^{r (-t_{1} - 1) - j},
      y^{r (-t_1) - j}
      $
      in all of the power series expansions at points $(\theta_i,0)$ of
      $
      \sum_{t = -t_2 + 1}^{-t_1} C\cdot G_t(x) y^{r t - j}
      $
      are integral.

		In particular, $C\cdot G_{-t_2 + 1}$ satisfies (iii) of Lemma
		\ref{lem:integralitytransfers}. Then the series expansions of $C\cdot
		G_{-t_2+1}(x)$ are all integral by condition (ii). Subtracting off $C\cdot G_{-t_2 +1}$, we see $C\cdot G_{-t_2 + 2}$ satisfies (iii) of Lemma \ref{lem:integralitytransfers}, hence condition (ii) and so on, so that all of the coefficients in all of the expansions of $\sum_{t = -t_2 + 1}^{-t_1} G_t(x) y^{r t - j}$ are integral. They remain integral upon differentiating.

		Rearranging \eqref{eqn:almostintegralityvertical}, the expansions of $C\cdot
		B(x) y^{-r t_1 + j}\dd x$ at each $(\theta_{i}, 0)$ as Laurent series in
		$\Q_{q}(\!(y)\!) \dd y$ are integral. Replacing $\dd y$ with $F'(x)y^{1-r}/r
		\dd x$ preserves integrality. A final application of Lemma
		\ref{lem:integralitytransfers} shows that $C\cdot B(x)$ is integral.
		\end{proof}

\section{Main algorithm} \label{sec:algorithm}

We now combine the techniques of the previous sections to compute the matrix
representing the $p$-th power Frobenius action with respect to
$\langle B_{\epsilon} \rangle \subset H^1 _\MW (\tildecalC)$  modulo $p^N$.
We summarize the procedure in Algorithm~\ref{algorithm}, where we take all intervals to be discrete, i.e., intersected with $\Z$.

\begin{algorithm}[htbp]
\DontPrintSemicolon
\SetAlgoVlined

  \SetCommentSty{textsc}
  \For{ $k \in [0, N - 1]$, $i \in [0, d - 2]$, $j \in [1 + \epsilon r , (1 + \epsilon)r-1]$, $\ell \in [0, d k + i + 1]$}
  {
      $T_{(i,j), k, \ell} \gets \mu_{j, k,  \ell - i - 1} x^{p \ell - 1} y^{- p (k r + j)}$; \tcp*[f]{ See Lemma~\ref{lemma:sparsefrob}}
  }
  \tcp{Horizontal reductions}
  \For{ $k \in [0, N - 1],\, j \in [1 + \epsilon r , (1 + \epsilon)r-1]$ }
  {
    $t \gets p(k r + j)$\;
    $L \gets \min(N - 1, d k + d - 2)$\;

    \tcp{Horizontal reductions modulo $p^N$, by linear recurrences}

    \For{ $\ell \in [0, L]$ }
    {
      $D(\ell), M(\ell) \gets D_H(p\ell, p(\ell + 1) - d - 1), M_H (p\ell, p(\ell + 1) - d - 1)$;
    }

    \tcp{Deduce the remaining $M(\ell)$ modulo $p^N$, by interpolation}

    \For{ $\ell \in [L + 1, d k + d - 2]$}
    {
      $D(\ell), M(\ell) \gets D_H(p\ell, p(\ell + 1) - d - 1), M_H (p\ell, p(\ell + 1) - d - 1)$
    }

    \tcp{Reduce $T_{(i,j), k}$ horizontally}

    \For{$i \in [0, d - 2]$}
    {
      $v \gets T_{(i,j),k, d k + i + 1}$;
      \tcp*{$v \in W_{p(d k + i + 1) - 1, t}$}

      \For{ $\ell = d k + i$ \KwTo $0$}
      {
        \For(\tcp*[f]{$W_{p(\ell + 1) -1, t} \rightarrow W_{p \ell - 1, t}$})
        {
          $e \in [1, d]$
        }
        {
          $v \gets D_H ^t (p (\ell + 1) - e)^{-1} \left( M_H ^t (p (\ell + 1) - e) \cdot v\right)$;
        }
        $v \gets T_{(i,j), k, l} + \bigl( D_H ^t (p \ell) ^{-1} M_H ^t (p \ell) \bigr)  \cdot \bigl( D(\ell)^{-1} M(\ell) \bigr) \cdot v$
      }
      $w_{(i,j), k} \gets v$; \tcp*{$w_{(i,j),k} \in W_{-1,t}$}
    }
  }
  \tcp{Vertical reductions}
  \For{ $j \in [1 + \epsilon r , (1 + \epsilon)r-1]$ }
  {
    \tcp{$p(kr + j) = r ( p k + \alpha) + \beta = p r (k + \lambda) + r \gamma + r \epsilon + \beta$}
    $\alpha, \beta \gets \lfloor pj / r \rfloor, pj \bmod r$
    \;
    $\lambda, \gamma \gets \lfloor (\alpha - \epsilon)/p \rfloor, (\alpha - \epsilon) \bmod r$
    \;
    $\delta \gets \gamma + \epsilon$
    \;
    \tcp{Vertical reductions modulo $p^{N+1}$, by linear recurrences}
    $M(0) \gets D_V ^\beta(\epsilon, \delta)^{-1} M_V ^\beta(\epsilon, \delta)$
    \;
    \For{$\ell \in [1, \lambda + N - 1]$}
    {
      $M(\ell) \gets D_V ^\beta (\delta + p (\ell - 1), \delta + p \ell) ^{-1} M_V ^\beta (\delta + p (\ell - 1), \delta + p \ell)$;
    }
    \For{$i \in [0, d - 2]$}
    {
      $v \gets w_{(i,j), N - 1 + \lambda}$ \tcp*{$v \in  V^{\beta} _{p (N - 1 + \lambda) +  \delta}$}
      \For(\tcp*[f]{$V ^\beta _{p k   + \delta } \rightarrow V ^\beta _{p (k - 1) + \delta }$}){
      $k = N - 1 + \lambda$ \KwTo $1$
      }{
        \eIf{$k \geq \lambda$}
        {
          $v \gets w_{(i,j), k - \lambda} + M(k) v$
          }{
          $v \gets M(k) v$
        }
      }
      $w_{(i,j)} \gets M(0) \cdot v$;
    }
  }
  \nl\KwRet{$w_{(i,j)}$, $i \in [0, d - 2]$, $j \in [1 + \epsilon r , (1 + \epsilon)r-1]$}
  \caption{Computes the matrix representing the $p$-th power Frobenius action with respect to
  $\langle B_{\epsilon} \rangle \subset H^1 _\MW (\tildecalC)$  modulo $p^N$}
  \label{algorithm}
\end{algorithm}

We now analyze the time and space complexity of Algorithm~\ref{algorithm}.
First, we recall that all our underlying ring operations are done in $\Z_q /
p^N$ or $\Z_q / p^{N+1}$.  Using bitstrings of length $O(N n \log p)$ to
represent elements of these rings, the basic ring operations (addition,
multiplication, and inversion) have bit complexity $\Otilde(N n \log p)$, the
matrix arithmetic operations on matrices of size $m \times m$  have bit
complexity $\Otilde(m^\omega N n \log p)$, and polynomial multiplication of
polynomials of degree $m$ has bit complexity $\Otilde(m N n \log p)$.
Applying Frobenius to such an element has complexity \(\Otilde(n\log^2 p + nN \log p)\) \cite[Corollary 3]{hubrechts-10}.

For $p$ sufficiently large, the dominant steps are the horizontal and vertical
reductions, i.e. lines $7$ and $23$ in Algorithm~\ref{algorithm}. In either
case, we apply a modification of \cite[Theorem 15]{bostan-gaudry-schost-07} to
achieve the $p^{1/2 + o(1)}$ time dependence.
\begin{proposition}[{Linear recurrences method, \cite[Theorem 6.1]{harvey-07}}]
  \label{prop:bsgs}
  Let $R = \Z_q/p^N$ or $\Z_q/p^{N+1}$, and $M(x) \colonequals M_0 + x M_1 \in R[x]^{m\times m}$.
  Let $0 \leq \alpha_1 < \beta_1 \leq \alpha_2 < \beta_2 \leq \dots \leq \alpha_h < \beta_h \leq K$ be integers.
  Assume $h < \sqrt{K} < p - 1$ and
  write $M(\alpha, \beta) \colonequals M(\alpha + 1) \cdots M(\beta)$.
  Then $M(\alpha_i, \beta_i)$ for $i = 1, \dots, h$ can be computed using $\Otilde(m^\omega \sqrt{K})$ ring operations in space $O(m^2 \sqrt{K})$.
\end{proposition}

For the horizontal reductions, we apply Proposition~\ref{prop:bsgs} once for
each pair $(k, j) \in [0, N - 1] \times [1 + \epsilon r, (1 + \epsilon) r - 1]$
with $K = O(p N)$ and $m = O(d)$.
For the vertical reductions, we apply Proposition~\ref{prop:bsgs} once for each
$j$, again with $K = O(p N)$ and $m = O(d)$.
This adds up to $\Otilde( p^{1/2} N^{3/2} r d^\omega)$ ring operations in space $O(p^{1/2} N^{1/2} d^2 )$.

Now we bound the time for the remaining steps. We will see that the number of
ring operations for the remaining steps is independent of $p$, so that they
contribute at most a $\log p$ term to the bit complexity.

To compute $\mu_{j, \ell, b}$ we start by replacing the coefficients of $F(x)$
by their images under $\sigma$.
We then calculate all $\sigma(F)^\ell _b$ in $O(d^2 N^2)$ ring operations.
Evaluating all the binomial coefficients and finding the $D_{j, \ell}$ uses  $O(r N^2)$ ring operations.
In total, we compute all the $\mu_{j,\ell, b}$ in $O( r d^2 N^2)$ ring operations plus $O(d)$ Frobenius substitutions.

We also use the $p$-adic interpolation method introduced by Harvey \cite[\S
7.2.1]{harvey-07} and attributed to Kedlaya.  This allows us to reduce the
number of matrix products that must be computed using the linear recurrence
algorithm.  The rest can then be obtained by solving a linear system involving a
Vandermonde matrix.  In our setting, an analogous complexity analysis holds, and
the total number of ring operations required is $O(r d^3 N^3)$, where the extra
$r$ factor is due to the $j$ loop.

The matrix $M_H ^t (s)$ is sparse; for each $t$, it requires $O(d)$ ring
operations to compute.  We need to do this $O(r N)$ times, thus the total is $O(
r d N)$.

During the horizontal reduction, we do the following for each $\ell$: $O(d)$ sparse
vector-matrix multiplications, and one dense vector-matrix multiplication. This requires $O(d^2)$ ring operations per $\ell$.  Hence, lines $10$-$16$ add up to $O(r d^4 N^2)$
ring operations.  The number of vector-matrix multiplications during the
vertical reduction is $O(d N)$, thus negligible in comparison with the
horizontal phase.

Computing all the $R_{i}$ and $S_{i}$ requires $O(d^3)$ total ring operations.
Then for each $j \in [r \epsilon + 1, (1 + \epsilon) r - 1]$, the matrix $M^j _V
(t)$ can be computed in $O(d^2)$ ring operations. The total number of
ring operations for these steps is $O(r d^2 + d^3)$.

The total number of operations is $O(p^{1/2} N^{3/2} r d^\omega + r d^4 N^3)$ plus $O(d)$ Frobenius substitutions.
Converting this to bit complexity, our algorithm runs in time 
\begin{equation}
  \label{eqn:main_time}
    \Otilde( p^{1/2} N^{5/2} r d^\omega n + N^4 r d^4 n \log p + N d n^2 \log p )\text{.}
\end{equation}

In addition to the space required by Proposition~\ref{prop:bsgs}, we use $O(r d^2 N)$ space for the interpolation, to store $w_{(i,j),k}$ and to do the vector-matrix multiplications.
This adds up to $O((p^{1/2} N^{3/2} + r N^2) d^2 n \log p)$ space, and Theorem~\ref{thm:mainthm} follows.

\begin{remark}
  \label{remark:low_space}
Under certain conditions, the time-space trade-off provided by Proposition~\ref{prop:bsgs} might not be ideal or possible.
  In those cases, one can instead do the reductions one step at a time with naive vector-matrix multiplications.
  The horizontal phase amounts to $O(p r d^2 N^2)$ sparse matrix-vector multiplications of size $O(d)$ in space $O(r d^2 N)$.
  The vertical phase amounts to $O(p r d N)$ dense matrix-vector multiplications of the same size, and no extra space is required.
  With the single exception of the $O(d)$ Frobenius substitutions, all the other steps are negligible in comparison.
  In terms of bit complexity, this amounts to $\Otilde(p r d^3 N^3 n + n^2 N \log p)$ time and $O(r d^2 N n \log p)$ space, and Theorem~\ref{thm:low_space} follows.
\end{remark}

\section{Sample Computations} \label{sec:sample_computations}
We have implemented both versions of our method using \textsc{SageMath}.
However, the  $p^{1/2 + o(1)}$ version, i.e., Theorem~\ref{thm:fast} and Algorithm~\ref{algorithm}, is only implemented for the case $n = 1$, as we rely on Harvey's implementation of Proposition~\ref{prop:bsgs} in \textsc{C++}.
Our implementation is on track to be integrated in one of the upcoming versions \textsc{SageMath} \cite{sagecode}.
An example session:
\begin{verbatim}
sage: x = PolynomialRing(GF(10007),"x").gen();
sage: CyclicCover(5, x^5 + 1).frobenius_polynomial()
x^12 + 300420147*x^8 + 30084088241167203*x^4 + 1004207356863602508537649
\end{verbatim}

Our examples were computed on one core of a desktop machine with an \texttt{Intel(R) Core(TM) i5-4590 CPU @ 3.30GHz}.
In all the examples, we took
\begin{equation}
  \label{eqn:N}
  N = \max \{ \lceil \log_p (4g/i) + n i /2 \rceil : i = 1, \dots, g \},
\end{equation}
and thus by employing Newton identities we can pinpoint the numerator of
$Z(\calC, t)$; see, for example, \cite[sl. 8]{kedlaya-oxfordtalk}.  In practice,
we may even work with lower $N$, and then hopefully verify that there is only
one possible lift that satisfies the Riemann hypothesis and the functional
equation in the Weil conjectures; see \cite{kedlaya-08}.

In Table~\ref{table:3tables} we present the running times for computing
$Z(\calC, t)$ for three examples where $(g, d, r) = (6, 5, 5), (25, 6, 12),$ and
$(45, 11, 11)$, over a range of $p$ values.  This sample of running times
confirms the practicality and effectiveness of our method for a wide range of
$p$ and tuples $(d,r)$.   We are not aware of any other alternative method that can handle $p$ and $g$
in these ranges.

\begin{table}[htbp]

\begin{subtable}{\textwidth}
  \centering
\begin{tabular}{lr|lr|lr}
  $p$ & time & $p$ & time & $p$ & time \\
  \hline
  $2^{14} - 3  $     &  \texttt{   1.21s}
                     &
  $2^{22} - 3  $     &  \texttt{   21.7s}
                     &
  $2^{30} - 35 $     &  \texttt{   5m58s}
  \\
  $2^{16} - 15 $     &  \texttt{   3.05s}
                     &
  $2^{24} - 3  $     &  \texttt{   40.9s}
                     &
  $2^{32} - 5  $     &  \texttt{  11m36s}
  \\
  $2^{18} - 5  $     &  \texttt{   5.74s}
                     &
  $2^{26} - 5  $     &  \texttt{   1m23s}
                     &
  $2^{34} - 41 $     &  \texttt{  32m59s}
  \\
  $2^{20} - 3  $     &  \texttt{   10.9s}
                     &
  $2^{28} - 57 $     &  \texttt{   2m54s}
                     &
  $2^{36} - 5  $     &  \texttt{    1h7m}
  \\
\end{tabular} 
\caption{Genus 6 curve 
$\calC \colon y^5 = x^{5} - x^{4} + x^{3} - 2x^{2} + 2x + 1$ with $N = 4$}
\label{table:c55}
\end{subtable}
\begin{subtable}{\textwidth}
  \centering
\begin{tabular}{lr|lr|lr}

$p$ & time & $p$ & time & $p$ & time \\
      \hline

$2^{10} + 45$ & \texttt{4m37s}
&
$2^{18} - 5 $      &  \texttt{   12m2s}
&
$2^{26} - 5 $      &  \texttt{   2h38m}
\\

$2^{12} - 3 $      &  \texttt{   5m31s}
&
$2^{20} - 3 $      &  \texttt{  21m34s}
&
$2^{28} - 57$      &  \texttt{   5h24m}
\\

$2^{14} - 3 $      &  \texttt{   6m20s}
&
$2^{22} - 3 $      &  \texttt{  37m21s}
&
$2^{30} - 35$      &  \texttt{  12h12m}
\\

$2^{16} - 15$      &  \texttt{   8m15s}
&
$2^{24} - 3 $      &  \texttt{   1h13m}
&
$2^{32} - 5 $      &  \texttt{  23h35m}
\\
\end{tabular}
\caption{Genus 25 curve 
$\calC \colon y^6 = x^{12} + 10x^{11} + x^{10} + 2x^{9} - x^{7} - x^{5} - 4x^{4} + 31x$ with $N = 13$}
\label{table:c612}
\end{subtable}
\begin{subtable}{\textwidth}
  \centering
  \begin{tabular}{lr|lr|lr}
    $p$ & time & $p$ & time & $p$ & time \\
    \hline
    $2^{12} - 3 $      &  \texttt{   24m1s}
                       &
    $2^{18} - 5 $      &  \texttt{    1h2m}
                       &
    $2^{24} - 3 $      &  \texttt{   7h21m}
    \\
    $2^{14} - 3 $      &  \texttt{  29m50s}
                       &
    $2^{20} - 3 $      &  \texttt{   1h52m}
                       &
    $2^{26} - 5 $      &  \texttt{  16h24m}
    \\
    $2^{16} - 15$      &  \texttt{  37m14s}
                       &
    $2^{22} - 3 $      &  \texttt{   3h22m}
                       &
    $2^{28} - 57$      & \texttt{ 33h17m}
    \\
  \end{tabular}
  \caption{Genus 45, 
  $\calC \colon y^{11} = x^{11} + 21x^{9} + 22x^{8} + 12x^{7} + 5x^{4} + 15x^{3} + 6x^{2} + 99x + 11$ with $N = 23$}
  \label{table:c1111}
\end{subtable}
\caption{Running times for three curves, for various $p$. Each subsequent row
represents a (roughly) four-fold increase in $p$ and a doubling in the running
time, confirming that our implementation has a $p^{1/2 + o(1)}$ running time.}
\label{table:3tables}
\end{table}

Our implementation is also favorable when compared with Minzlaff's implementation (in \textsc{Magma 2.24-1}), which deals only with superelliptic curves, rather than arbitrary cyclic covers.
For example, consider the superelliptic curve
$$C\colon y^7 =  x^{3} + 4 \, x^{2} + 3 \, x - 1\text{,}$$
if we wish to compute all $L$ polynomials of $C$ for $p < 2^{24}$ using our implementation we estimate that this will take about 6 months on one core (on the same desktop mentioned above),
 whereas with Minzlaff's it would take around 3 years.
The curve $C$ has some interesting properties and it arose recently in some in progress work of D. Roberts, F. Rodriguez-Villegas, and J. Voight.

\bibliographystyle{alpha}
\bibliography{biblio}

\end{document}